\documentclass[12pt,a4paper]{amsart}  
\input{preamble.sty}
\begin{document}
\title{On extension of Green's operator on bounded smooth domains}
\author{ Antti V. V\"ah\"akangas}
\address{Department of Mathematics and Statistics,
P.O. Box 68 (Gustaf H\"allstr\"omin katu 2b),
FI-00014 University of Helsinki, Finland}
\email{antti.vahakangas@helsinki.fi}
\date{}

\begin{abstract}
We prove a regularity result for Green's functions that are associated to elliptic second
order divergence-type linear PDO's with coefficients in  $C^{1,\alpha}(\overline{\Omega})$. 
Here $\alpha\in (0,1)$ and $\Omega\subset \R^n$ is a bounded $C^{2,\alpha}$ domain in
dimension $n\ge 3$.
The regularity result gives boundary estimates for derivatives up to order $(2+\alpha)$ and, 
by using these estimates, we 
extend the associated Green's operator to a globally
defined singular integral  which of Calder\'on--Zygmund type.
\end{abstract}
\maketitle

\section{Introduction}

\subsection{Background}
Let
$\Omega\subset\R^n$, $n\ge 3$,  be a bounded domain satisfying an exterior ball condition, and
consider the boundary value problem:
\begin{equation}\label{dir}
\begin{cases}
-L u=f\in L^2(\Omega),\\
u\in W^{1,2}_0(\Omega).
\end{cases}
\end{equation}
Here $L$ is a second order partial differential operator, which is of divergence type,
\begin{equation}\label{alla}
Lu=\sum_{i,j=1}^n \partial_i(a^{ij}\partial_j u)
\end{equation}
such that the coefficients $a^{ij}\in L^\infty(\Omega)$ are symmetric
($a^{ij}=a^{ji}$) and $L$ is strictly elliptic: there is a constant $\lambda>0$ such that
for almost every $x\in\Omega$, we have
\begin{equation}\label{elliptic}
\lambda |\xi|^2\le \sum_{i,j=1}^n a^{ij}(x)\xi_i\xi_j\quad \forall\,\xi\in\R^n.
\end{equation}
A prototype  is the Laplacian $L=\Delta=\sum_{i=1}^n \partial_i^2$ for which
the problem \eqref{dir} in case of domains with only Lipschitz boundary
is studied in \cite{jerison}.

It is well known that the solution of \eqref{dir}  can be expressed in terms of a so called Green's operator:
\begin{equation}\label{sol}
u(x)=\mathcal{G}f(x)=\int_\Omega G(x,y)f(y)dy,\quad x\in \Omega.
\end{equation}
The existence of Green's operator is established in the fundamental paper \cite{gunther}. In what follows we recapitulate some results therein.

The following existence result is of importance
to us.

\begin{thm}\label{gext}
There exists a unique function 
$G:\Omega\times \Omega\setminus\{(x,x)\}\to \R$, $G\ge 0$, such that for every
$x\in\Omega$,
\[
G(x,\cdot)\in W^{1,1}_0(\Omega)\cap W^{1,2}(\Omega\setminus B(x,r))
\]
and also, if $\phi\in C^\infty_0(\Omega)$, then
\[
\langle -LG(x,\cdot),\phi\rangle = \sum_{i,j=1}^n\int_\Omega a^{ij}(y)\partial_{y_j} G(x,y) \partial_i \phi(y)dy = \phi(x),\quad x\in\Omega.
\]
The function $G$ is the Green's function for the operator $-L$. It satisfies
$G(x,y)=G(y,x)$ and $G(x,y)\le C|x-y|^{2-n}$ if $x,y\in\Omega$.
\end{thm}

A proof can be found in \cite{gunther}.
Since $\Omega$ satisfies, in particular, an exterior cone condition, the
Green's function has H\"older regularity even if the coefficients $a^{ij}$ are only
essentially bounded, see \cite[Theorem 1.8, Theorem 1.9]{gunther}.
More regularity is available if one assumes that the coefficients
are Dini continuous,
\begin{equation}\label{dini}
|a^{ij}(x)-a^{ij}(y)|\le \omega(|x-y|),\quad x,y\in\Omega.
\end{equation}
Here $\omega:\R_+\to \R_+$ is supposed to be non-decreasing, $\omega(2t)\le K\omega(t)$ for 
some $K>0$ and all $t>0$
and
\[
\int_0^1 \frac{\omega(t)}{t}dt < \infty.
\]
In case the coefficients belong to the space $C^{\alpha}(\overline{\Omega})\supset C^{1,\alpha}(\overline{\Omega})$, they
satisfy the Dini condition. The following is proven 
in \cite[Theorem 3.3]{gunther}.

\begin{thm}\label{gu}
Assume 
that \eqref{elliptic} holds and  the coefficients $a^{ij}$ are Dini continuous.
Then the Green's function of the corresponding differential operator $-L$ satisfies the following
inequalities for any $x,y\in\Omega$; here $\delta(\cdot)=\mathrm{dist}(\cdot,\partial\Omega)$,
\begin{align*}
&G(x,y)\le C|x-y|^{2-n}\min\bigg\{1,\frac{\delta(x)}{|x-y|},\frac{\delta(x)\delta(y)}{|x-y|^2}\bigg\};\\
&|\nabla_y G(x,y)|\le C|x-y|^{1-n}\min\bigg\{1,\frac{\delta(x)}{|x-y|}\bigg\}.
\end{align*}
Furthermore, the  mixed derivatives satisfy $|\nabla_x\nabla_y G(x,y)|\le C|x-y|^{-n}$, $x,y\in\Omega$.
\end{thm}

The following estimate for more regular coefficients is in \cite[Theorem 1.8]{fassihi}.

\begin{thm}\label{ssaa}
Assume 
that \eqref{elliptic} holds and the coefficients $a^{ij}$ belong to  $C^{1,\alpha}(\overline{\Omega})$
for some $\alpha\in (0,1)$. Then
\begin{align*}
|\nabla_x^2\nabla_y G(x,y)|\le C|x-y|^{-n}/\min\{|x-y|,\delta(x)\}
\end{align*}
for every $x,y\in\Omega$.
\end{thm}


Estimates like above have been used in
establishing weak type $(1,1)$ estimates for operators $\nabla^2 \mathcal{G}$ on
bounded and convex domains, $n\ge 3$.
In case of Laplacian this is an unpublished result due to Dahlberg, Verchota, and Wolff;
a proof can be found in \cite{fromm}. In case of Lipschitz coefficients
similar methods are shown to apply \cite{fassihi}.
The weak type estimate is established by
utilizing theory of singular integrals in $\R^n$ -- first one
extends kernel $\nabla_x^2 G(x,y)$ by zero
and then proves that extended kernel $K=\chi_{\Omega\times\Omega}\nabla_x^2 G$ satisfies
the H\"ormander condition
\begin{equation}\label{hor}
\int_{|x-y|\ge 5|h|}Ê|K(x,y+h)-K(x,y)|dx \le C
\end{equation}
with $C$ independent of $y,h\in\R^n$.

In this paper we study when $\nabla^2 \mathcal{G}$ extends to
a Calder\'on--Zygmund operator \cite{D-J}. These are 
linear operators $T\in\mathcal{L}(L^2(\R^n))$ having a
kernel representation
\[
\langle Tf,g\rangle =\int_{\R^n} \int_{\R^n} K(x,y)f(y)g(x)dydx
\]
if the supports of $f,g\in C^\infty_0\subset\mathcal{S}$ are disjoint.
In this connection it is assumed that $K$ is a so called Calder\'on--Zygmund kernel:
there is $\delta\in (0,1]$ such that
\begin{equation}\label{standard}
\begin{split}
&|K(x,y)|\le C|x-y|^{-n},\quad x,y\in\R^n,\\
&|K(x,y+h)-K(x,y)|\le C|h|^{\delta}|x-y|^{-n-\delta},\quad |h|\le |x-y|/2,
\end{split}
\end{equation}
and the transposed kernel $K^t(x,y)=K(y,x)$ also satisfies  estimates \eqref{standard}.
It is easy to verify that $K$ satisfies the H\"ormander condition \eqref{hor}, so
the operator $T$ is of weak type $(1,1)$ and, also,
bounded on $L^p(\R^n)$ for $p\in (1,\infty)$.

In light of estimates \eqref{standard} it is a prerequisite for the extension of $\nabla^2 \mathcal{G}$
that one can
control the derivatives of $K=\nabla^2_x G(x,y)$ up to order $\le 2+\delta$ and up to the boundary.
To indicate this, assume that $\Omega$ is convex and bounded. We can
estimate the H\"older condition in \eqref{standard} by
using the mean-value theorem and Theorem \ref{ssaa}. Accordingly,  there
is $\xi\in [y,y+h]\subset\Omega$ for which
\[
|K(x,y+h)-K(x,y)|\le |h||\nabla_y K(x,\xi)| \le C|h||x-y|^{-n}/\min\{|x-y|,\delta(x)\}.
\]
The upper bound blows up
when $x$ tends to the boundary. In particular, extension of $K$ to a 
Calder\'on--Zygmund kernel is not possible by using these estimates.

Boundary estimates for higher order derivatives of  Green's functions, associated to uniformly elliptic operators of order $2m$
on bounded domains, are derived in \cite{krasovskii}
and later refined in \cite{sweers}. 
Therein, if $m=1$ (as in our case) and dimension $n\ge 3$, the coefficients of $L$ should belong to the space
$C^{5}(\overline{\Omega})$ and  $\partial\Omega$ to the  class $C^7$.

\subsection{Main results}
First we establish (order $2+\alpha$) boundary regularity estimates for Green's functions under reasonable assumptions.
Using these we then extend the corresponding Green's operator
to a Calder\'on--Zygmund operator.

Througout this section we will assume that
 $\Omega\subset \R^n$, $n\ge 3$, 
is a bounded $C^{2,\alpha}$ domain, $\alpha\in (0,1)$. In particular, this domain  satisfies
an exterior ball condition.
We will also assume 
that the coefficients
$a^{ij}$ are $C^{1,\alpha}(\overline{\Omega})$ regular, and we will
denote \[\delta(\cdot)=\mathrm{dist}(\cdot,\partial\Omega).\]

Here is our boundary regularity result for derivatives up to order $2+\alpha$:

\begin{thm}\label{dss}
Assume 
that  $\Omega\subset \R^n$
and the coefficients $a^{ij}$ are as above.
Then, if $\beta\in\N_0^n$ satisfies $|\beta|\le 2$, we have
\begin{equation}\label{tapa}
\begin{split}
|\partial^\beta_y G(x,y)|\le C|x-y|^{2-n-|\beta|}\min\bigg\{1,\frac{\delta(x)}{|x-y|}\bigg\},\quad x,y\in\Omega.
\end{split}
\end{equation}
Furthermore, if $|\beta|=2$, then
\begin{equation}\label{hol}
|\partial_y^{\beta} G(x,y+h)-\partial_y^\beta G(x,y)|\le C|h|^\alpha |x-y|^{-n-\alpha}\min\bigg\{1,\frac{\delta(x)}{|x-y|}\bigg\}
\end{equation}
for every $x,y,y+h\in\Omega$ satisfying $|h|\le |x-y|/2$. 
\end{thm}

If $|\beta|<2$, estimate
\eqref{tapa} is  covered in Theorem \ref{gu}.
Under further regularity assumptions for the coefficients and domain, estimates like \eqref{tapa} 
for higher order derivatives are established in
\cite[Theorem 12]{sweers}.

The proof of Theorem \ref{dss} relies on the (known) size estimate
\[
G(x,y)\le C|x-y|^{2-n}\min\bigg\{1,\frac{\delta(x)}{|x-y|}\bigg\},\quad x,y\in\Omega,
\]
and certain
local boundary type Schauder estimates
\cite{gilbarg}.
Latter estimates are available on bounded $C^{2,\alpha}$ smooth domains, and they give
us control to the solutions of $Lu=0$ up to the boundary of the domain
and for derivatives up to order $2+\alpha$.

We then utilize the refined estimates given in Theorem \ref{dss} by showing that
the Green's operator extends to an integral operator in $\R^n$, whose
second order partials are Calder\'on--Zygmund type operators.
For this purpose we will invoke various results about weakly singular integral
operators on domains whose theory is developed in the thesis \cite{t1dom}.
First we invoke following spaces.

\begin{defn}\label{stand}
Let $\emptyset\not=D\subset\R^n$, $n\ge 2$, be a  
domain. Assume that  $m\in\N$ and $0<\delta<1$.
The space of {\em smooth kernels}, denoted by
$\mathcal{K}^{-m}_D(\delta)$, consists 
of complex-valued functions 
$K\in C^m(D\times D\setminus\{(x,x)\})$ satisfying
\begin{itemize}
\item
 size-estimate, given $\alpha,\beta\in\N_0^n$ so that $|\alpha|+|\beta|\le m$,
\[
|\partial^\alpha_x\partial^\beta_y K(x,y)|
\le C_K|x-y|^{m-n-|\alpha|-|\beta|},\quad x,y\in D.
\]
\item
 H\"older-regularity estimate, given  $\alpha,\beta\in\N_0^n$ so that
$|\alpha|+|\beta|=m$,
\begin{align*}
|\partial^\alpha_x\partial^\beta_y K(x,y+h)-\partial^\alpha_x\partial^\beta_y K(x,y)|
\le C_K|h|^{\delta}|x-y|^{-n-\delta}
\end{align*}
if $x,y,y+h\in D$ satisfy $|h|\le|x-y|/2$. We also
assume the same estimate with $h$-difference placed to the $x$-variable and  $x,y,x+h\in D$ satisfying $|h|\le|x-y|/2$.
\end{itemize}
\end{defn}

It is important to observe that the order $m$ derivatives
of kernels in $\mathcal{K}_{\R^n}^{-m}(\delta)$ are Calder\'on--Zygmund kernels,
that is, they satisfy estimates \eqref{standard}.
This follows from the case $D=\R^n$ and $|\alpha|+|\beta|=m$ in Definition \ref{stand}.

We show that Green's function belongs to the space $\mathcal{K}^{-2}_\Omega(\alpha)$. 
Furthermore, 
as a main result of this paper, we will establish the following extension theorem.

\begin{thm}\label{mainth}
Let $\Omega$ and $a^{ij}$ be as quantified above. Then
there exists a smooth kernel $\hat G\in\mathcal{K}^{-2}_{\R^n}(\alpha)$ such that
\[
\hat G|\Omega\times\Omega\setminus\{(x,x)\} = G,
\]
and the operators $\partial^\sigma \hat{\mathcal{G}},\partial^\sigma\hat{\mathcal{G}}^*$, $|\sigma|=2$, are Calder\'on--Zygmund  operators.
Hence they belong to $\mathcal{L}(L^p(\Omega))$ for $1<p<\infty$.
Here $\hat{\mathcal{G}}$ denotes the integral operator  associated with $\hat G$, that is, \[\hat{\mathcal{G}}f(x)=\int_{\R^n} \hat G(x,y)f(y)dy,\quad f\in C^\infty_0(\R^n).\]   
It follows that the differentiated Green's operators
$\partial^\sigma \mathcal{G}$, $|\sigma|=2$, are
restrictions of Calder\'on--Zygmund operators to the domain $\Omega$,
\[
\langle \mathcal{G}f,\partial^\sigma g\rangle = \langle \hat{\mathcal{G}}f,\partial^\sigma g\rangle,\quad f,g\in C^\infty_0(\Omega).
\]
\end{thm}

This paper is organized as follows: in Section \ref{two} we define weakly singular integral operators 
and invoke their basic properties from \cite{t1dom}.
We also prove a sharpening of one of the results in \cite{t1dom}; this sharpening is needed.
Section \ref{three} begins with
Schauder estimates, taken from \cite{gilbarg}, and it ends with proofs
of Theorem \ref{dss} and Theorem \ref{mainth}.

\section{Weakly singular integral operators}\label{two}

Here we recapitulate theory of weakly singular integral operators on domains,
developed in \cite{t1dom}.
First we define bounded $C^{2,\alpha}$ domains, but also uniform
and coplump domains. The latter classes of domains are useful in connecting to 
theory of weakly singular integrals.

\subsection{Classes of domains}
 The
Green's function will be defined on a bounded $C^{2,\alpha}$ domain.
For later purposes we need
to verify that such a domain is both uniform and complump. This type of 
results are well known, so we only indicate the proofs.


\begin{defn}\label{smoothdom}
A bounded domain $\Omega\not=\emptyset$ in $\R^n$ is of class $C^{2,\alpha}$,
$0<\alpha<1$, if at each point $\bar y\in \partial \Omega$ there
is a ball $B=B(\bar y,\rho(\bar y))$ and a diffeomorphism
$\psi:B\to D\subset\R^n$ such that the following conditions (1)--(3) hold:
\begin{itemize}
\item[(1)] $\psi(B\cap \Omega)\subset \R^n_+$;
\item[(2)] $\psi(B\cap \partial \Omega)\subset \partial \R^n_+$;
\item[(3)] $\psi\in C^{2,\alpha}(\bar B), \psi^{-1}\in C^{2,\alpha}(\bar D)$.
\end{itemize}
In (3) we assume that norms
of the diffeomorphisms $\psi$ are uniformly bounded
by some constant $K=K(\Omega)>0$,
\[||\psi||_{C^{k,\alpha'}(\bar B)}+||\psi^{-1}||_{C^{k,\alpha'}(\bar D)}\le K,\quad k+\alpha'\le 2+\alpha.\]
 In particular, the
Lipschitz constants of $\psi$ and 
$\psi^{-1}$ bounded by $K$.
We shall
say that the diffeomorphism $\psi$ straightens the boundary near $\bar y$.
\end{defn}

That a bounded domain is of class $C^{2,\alpha}$
is a local property of its boundary.
If $\Omega$ is such a domain then, due to compactness of $\partial\Omega$, there exists $\rho>0$
and a set $\{\bar y_1,\ldots,\bar y_k\}\subset \partial \Omega$
such that, if
$\bar y\in\partial\Omega$ is a boundary point, then
\[B(\bar y,\rho)\subset \subset B(\bar y_{\ell},\rho(\bar y_\ell))\]
for some $\ell\in \{1,2,\ldots,k\}$. 
In the sequel we assume that $\rho$ satisfies $\rho/\mathrm{diam}(\Omega)<1/4$.

\begin{defn}\label{unif}
A  domain $\Omega\subset \R^n$, $n\ge 2$, is {\em $c$-uniform} for $c\ge 1$ if every pair of distinct points $x,y\in\Omega$
can be joined by a $c$-cigar in $\Omega$, that is,
there exists a continuum $E\subset \Omega$ containing these two points such that
$\mathrm{diam}(E)\le c|x-y|$ and 
\[\min\{|z-x|,|z-y|\}\le c\mathrm{dist}(z,\partial\Omega)\]
if  $z\in E$.
\end{defn}

In \cite[Definition 1.13]{t1dom} we
pose a definition which is based on rectifiable paths, but this is equivalent to the definition given above
\cite{vaisala}.

\begin{defn}\label{plump}
A domain $\Omega\subset \R^n$, $n\ge 2$,
is $c$-{\em coplump}, $c\ge 1$, if for all $x\in \R^n\setminus\Omega$
and $0<r<\mathrm{diam}(\R^n\setminus\Omega)$ there is $z\in \bar{B}(x,r)$ with
$B(z,r/c)\subset \R^n\setminus \Omega$. 
\end{defn}

Uniformity and coplumpness are local properties of the boundary of the domain. 
This is  easily seen for the coplumpness, and we omit the precise
formulation. For uniformity we invoke the following theorem,  
due to J. V\"ais\"al\"a \cite[Theorem 4.1]{vaisala}.

\begin{thm}
Suppose that $\Omega\subset\R^n$ is a bounded domain and 
that $c\ge 1$, $0<r<\mathrm{diam}(\Omega)$. Suppose also
that if $z\in\partial\Omega$, then every pair of points
in $\Omega\cap B(z,r)$ can be joined by a $c$-cigar in $\Omega$.
Then $\Omega$ is $c_1$-uniform with $c_1=40c^2\mathrm{diam}(\Omega)/r$.
\end{thm}

Here is a simple consequence of the locality of all definitions given above 

\begin{thm}\label{smoothun}
Let $\Omega\subset\R^n$, $n\ge 2$, be a 
 bounded $C^{2,\alpha}$ domain. Then
 \begin{itemize}
\item $\Omega$ is $c$-uniform with $c=200c_n^3K^6\mathrm{diam}(\Omega)/\rho$,
where  $c_n$ is any constant such that
$\R^n_+$ is $c_n$-uniform.
\item $\Omega$ is $c$-coplump for $c=8K^4\mathrm{diam}(\Omega)/\rho$.
\end{itemize}
\end{thm}

We omit the proof. It relies on the bi-Lipschitz bound $K$ of the diffeomorphisms
straightening the boundary near the boundary points. Hence the result
holds true if we only assume that $\Omega$ is a bounded
Lipschitz domain, that is, the mappings $\psi$ are only assumed
to be bi-Lipschitz.

\subsection{Standard kernels and their regularity}
A so called standard kernel space furnishes an approximation theoretic approach to
smooth kernels if the underlying domain is uniform. This space is defined in terms of integral averages of differences, and its
advantages include that it is -- a priori -- easier to verify that a given kernel is standard
than smooth.

The difference operators $y\mapsto \Delta_h^\ell(f,D,y):\R^n\to \C$ are parametrized by $\ell\in\N$, $h\in\R^n$, and $D\subset\R^n$.
These operate on functions $f:D\to \C$ according to the rule
\[
\Delta_h^{\ell}(f,D,y)=
\begin{cases}
\sum_{k=0}^\ell (-1)^{\ell+k}\binom{\ell}{k} f(y+kh), & \text{ if } 
\{y,y+h,\ldots,y+\ell h\}\subset D,
\\
0,&\text{otherwise}.
\end{cases}
\]
\begin{defn}
Let $\emptyset\not=\Omega\subset \R^n$, $n\ge 2$, be a domain. Let
$m\in\N$  and $0<\delta< 1$. The space of {\em standard kernels}, denoted by $\mathrm{K}^{-m}_\Omega(\delta)$, 
consists of continuous functions $K:\Omega\times \Omega\setminus\{(x,x)\}\to \C$ 
satisfying
\begin{itemize}
\item  kernel size estimate
\begin{equation}\label{k1}
|K(x,y)|\le C_K|x-y|^{m-n},\quad x,y\in\Omega,
\end{equation}
 \item semilocal integral estimate
\begin{equation}\label{k2}
\sup_{|h|\le \mathrm{diam}(B)}
\frac{1}{|B|^{1+(m+\delta)/n}}\int_{B}
|\Delta_h^{m+1}(K(x,\cdot),B,y)|dy \le
C_K|x-y^B|^{-n-\delta},
\end{equation}
if $x\in \Omega$ and $B\subset\subset\Omega$ is a ball, centered at $y^B$ 
so that
$C_K\mathrm{diam}(B)\le|x-y^B|$.
 We also assume the  estimate \eqref{k2} with $K(x,\cdot)$
 replaced by $K(\cdot,x)$. 
\end{itemize}
\end{defn}

Notice that we use balls in the definition of standard kernels instead of cubes, but this difference
compared to the definition given in \cite{t1dom} is not important.
Here is a result stating that standard kernels are smooth if the underlying domain is uniform.

\begin{thm}\label{fixss}
Let $\Omega\subset \R^n$ be a uniform domain
and  $0<m<n$ and $0<\delta<1$. Then the classes of standard and smooth kernels coincide, that is,
\[
\mathrm{K}^{-m}_{\Omega}(\delta) = \mathcal{K}^{-m}_{\Omega}(\delta).
\]
As a consequence, if $K\in\mathrm{K}^{-m}_{\R^n}(\delta)$ and
$\alpha,\beta\in \N_0^n$ satisfy
$|\alpha|+|\beta|=m$, then 
\[\partial^\alpha_x\partial^\beta_y K:\R^n\times \R^n\setminus\{(x,x)\}\to \C\] is
a Calder\'on--Zygmund kernel. 
\end{thm}

A related result -- where the sharp H\"older exponent is missing -- is proven in \cite[Chapter 4]{t1dom} by using approximation
theoretic approach. In what follows we modify that approach, and thereby
provide a proof for Theorem \ref{fixss}.
Generally speaking, the proof is based on
so called dyadic resolution of unity on uniform domains.
Let us explain what we mean by this:
In case $\Omega=\R^n$, we can fix
one bump function $\phi\in C^\infty_0(\R^n)$ so that
$\int_{\R^n} \phi(x)dx=1$ and define
$\phi_{j}=2^{jn}\phi(2^{j}\cdot)$ for $j\in\Z$. 
This results in the decomposition 
\begin{equation}\label{repro}
f(x) =\int_{\R^n} f(y)\phi_{\ell}(y-x)dy + \sum_{j=\ell+1}^\infty \int_{\R^n} f(y)(\phi_{j}-\phi_{j-1})(y-x)dy,\quad \ell\in\Z,
\end{equation}
assuming, say,  that $f$ is continuous at $x$.

In proper domains the  difficulties
lie in modifying this
construction such that the supports of the bump functions are included in
the domain and the coarseness parameter $\ell$ is independent of 
$\delta(x)$. To indicate 
the difficulties, 
one expects  vanishing moments  from
the difference of two consecutive bump functions in order to induce cancellation. There
are also certain geometric properties that the modification should preserve.


We begin the proof of Theorem \ref{fixss} with invoking a dyadic resolution of a given kernel $K\in\mathrm{K}^{-m}_{\Omega}(\delta)$ in
a uniform domain $\Omega\subset\R^n$ \cite[p. 69]{t1dom}.
Let $x_0,y_0\in \Omega$ be distinct points and
let $\ell=\ell(x_0,y_0)$ be defined by
$2^{-\ell}<|x_0-y_0|/16<2^{-\ell+1}$. Then we let
$\{\phi_{\sigma,2^{-j}}\}_{j\ge \ell},\{\phi_{\rho,2^{-j}}\}_{j\ge \ell}$ be $m$-regular bump functions
along so called quasihyperbolic geodesics $\sigma:x_0\curvearrowright y_0$ and $\rho=\sigma^{-1}:y_0\curvearrowright x_0$.
These bump functions approximate the Dirac's delta at the origin
and satisfy $\mathrm{supp}\,\phi_{\sigma,2^{-j}}(\cdot-x)\subset\Omega$
if $x$ is close to $x_0$. The geometry of uniform domains plays a crucial role
in the  construction of these functions.

Denote
\[\Omega(x_0,y_0)=B(x_0,r(x_0)\wedge (|x_0-y_0|/4b)),\quad r(x_0)=\mathrm{dist}(x_0,\partial\Omega)/4b,\]
where the constant $b\ge 1$ depending on the geometry of the domain is given in Lemma \cite[Lemma 4.16]{t1dom}.
Then, if $j,k\ge \ell=\ell(x_0,y_0)$
and $(x,y)\in \Omega(x_0,y_0)\times \Omega(y_0,x_0)\subset \Omega\times\Omega$, we denote
\begin{equation}\label{denoteds}
K_{j}^{\sigma,\rho}(x,y)=\int_\Omega \phi_{\sigma,2^{-j}}(\alpha-x)\int_{\Omega}
K(\alpha,\omega)\phi_{\rho,2^{-j}}(\omega-y)d\omega d\alpha.
\end{equation}
Due to continuity of the kernel in the domain $\Omega\times\Omega\setminus\{(x,x)\}$, we have  the following decomposition, valid for the
points $(x,y)\in \Omega(x_0,y_0)\times\Omega(y_0,x_0)$,
\begin{equation}\label{hajotelma1}
K(x,y)=\lim_{j\to\infty} K_j^{\sigma,\rho}(x,y)=K_\ell^{\sigma,\rho}(x,y)+\sum_{j=\ell+1}^\infty 
\big(K_{j}^{\sigma,\rho}-K_{j-1}^{\sigma,\rho}\big)(x,y).
\end{equation}
For notational purposes it is convenient to 
denote $\kappa^{\sigma,\rho}_\ell=K_\ell^{\sigma,\rho}$ and
express the differences inside the summation in the following way. Given
$j>\ell$ and points $x,y$ as above, we also denote 
\begin{equation}\label{munu}
\begin{split}
\kappa_j^{\sigma,\rho}(x,y)&:=K_j^{\sigma,\rho}(x,y)-K_{j-1}^{\sigma,\rho}(x,y) \\&= \int_\Omega
(\underbrace{\phi_{\sigma,2^{-j}}-\phi_{\sigma,2^{-j+1}}}_{=:\psi_{\sigma,j}})(\alpha-x)\int_\Omega K(\alpha,\omega)
\phi_{\rho,2^{-j}}(\omega-y)d\omega d\alpha\\
&\quad +\int_\Omega
\phi_{\sigma,2^{-j+1}}(\alpha-x)\int_\Omega K(\alpha,\omega)
(\underbrace{\phi_{\rho,2^{-j}}-\phi_{\rho,2^{-j+1}}}_{=:\psi_{\rho,j}})(\omega-y)d\omega d\alpha
\\&=:\mu^{\sigma,\rho}_j(x,y) + \nu^{\sigma,\rho}_j(x,y).
\end{split}
\end{equation}
As a consequence, we can write
\begin{equation}\label{hajotelma}
K(x,y)=\sum_{j=\ell}^\infty \kappa^{\sigma,\rho}_j(x,y),\quad (x,y)\in \Omega(x_0,y_0)\times\Omega(y_0,x_0).
\end{equation}

The proof of the following auxiliary lemma is essentially the same
as the proof of \cite[Lemma 4.33]{t1dom}. One of the important
ingredients is that the differences of consecutive bump functions satisfy 
\[\int_{\R^n} x^\alpha \psi_{\sigma,j}(x)dx=0=\int_{\R^n} x^\alpha \psi_{\rho,j}(x)dx,\quad |\alpha|\le m,\]
and one of these differences appear in the definition of both $\mu^{\sigma,\rho}_j$ and $\nu^{\sigma,\rho}_j$.
This allows one to connect to the integral
estimate \eqref{k2}, satisfied by standard kernels and their transposes.

\begin{lem}\label{apu2}
Let $\Omega\subset\R^n$ be a uniform domain and
$T\in \mathrm{SK}^{-m}_\Omega(\delta)$
be associated with a kernel $K\in\mathrm{K}^{-m}_\Omega(\delta)$ that
is decomposed as in \eqref{hajotelma}. 
Let $j\ge \ell=\ell(x_0,y_0)$.
Then the summands in this decomposition enjoy the regularity  \[\kappa_{j}^{\sigma,\rho}\in C^\infty(\Omega(x_0,y_0)\times\Omega(y_0,x_0))\] and, if $\alpha,\beta\in \N_0^n$ and
$(x,y)\in \Omega(x_0,y_0)\times\Omega(y_0,x_0)$, they satisfy the estimate
\begin{equation}\label{bada}
|\partial^{\alpha}_x\partial^\beta_y \kappa_{j}^{\sigma,\rho}(x,y)|\le 
C2^{j(|\alpha|+|\beta|-m-\delta)}|x-y|^{-n-\delta},
\end{equation}
where the constant $C$ depends at most on  
$n,m,\alpha,\beta,K,\Omega$.
\end{lem}

We have performed all the preparations, and we can
proceed to the actual proof of Theorem \ref{fixss}.
It will be a straightforward
modification of the proof of Theorem \cite[Theorem 4.37]{t1dom}, and we
only give the required modifications.

Let $x_0,y_0\in \Omega$ be distinct   and  $\ell=\ell(x_0,y_0)$. 
Let $(x,y)\in \Omega(x_0,y_0)\times\Omega(y_0,x_0)$
and $|\alpha|+|\beta|\le m$. Then $|x-y|\ge|x_0-y_0|/2$ and,
combining this estimate with Lemma \ref{apu2}, we obtain
\begin{equation}\label{es}
\begin{split}
&\sum_{j=\ell}^\infty 
|\partial^\alpha_x \partial^\beta_y \kappa_{j}^{\sigma,\rho}(x,y)|
\le 
C|x_0-y_0|^{-n-\delta}\sum_{j=\ell}^\infty
2^{j(|\alpha|+|\beta|-m-\delta)}
\le C |x_0-y_0|^{m-n-|\alpha|-|\beta|}.
\end{split}
\end{equation}
The Weierstrass $M$--test, combined with the identity \eqref{hajotelma}, shows
that \begin{equation}\label{esita}
K|(\Omega(x_0,y_0)\times \Omega(y_0,x_0))=\sum_{j=\ell}^\infty \kappa_{j}^{\sigma,\rho}\in C^m(\Omega(x_0,y_0)\times \Omega(y_0,x_0)),\end{equation} and the series
 can be differentiated termwise up to the order $m$.
As a consequence of this identity we have the regularity
$K\in C^m(\Omega\times\Omega\setminus\{(x,x)\}$ and, by using  \eqref{es}, we also have the estimate
\begin{equation}\label{derivata}
|\partial^\alpha_x\partial^\beta_y K(x_0,y_0)|\le
C|x_0-y_0|^{m-n-|\alpha|-|\beta|},\quad |\alpha|+|\beta|\le m,
\end{equation}
which is the required size-estimate for smooth kernels.

We turn to the required H\"older-regularity estimates.
Due to symmetry it suffices to consider differences in the first $\R^n$-variable only.
To begin with consider the situation, where $x_0,y_0\in\Omega$ are distinct points,  $|\alpha|+|\beta|=m$, and 
$h\in \R^n$ is close to $x_0$ so that $x_0+h\in \Omega(x_0,y_0)$.
Fix $j_0\in\Z$ such that $2^{-j_0}<|h|\le 2^{-j_0+1}$.

Fix $j\ge \ell=\ell(x_0,y_0)$ and  denote 
\[
\Delta_h^1(\partial^\alpha_x\partial^\beta_y \kappa_{j}^{\sigma,\rho}(\cdot,y_0),x_0)=\partial^\alpha_x \partial^\beta_y \kappa_{j}^{\sigma,\rho}(x_0+h,y_0)-
\partial^\alpha_x \partial^\beta_y \kappa_{j}^{\sigma,\rho}(x_0,y_0).
\]
Applying the mean value theorem and Lemma \ref{apu2} we find 
a point $\xi\in\R^n$, belonging to the line segment $[x_0,x_0+h]\subset \Omega(x_0,y_0)$,
so that $|\xi-y_0|\ge |x_0-y_0|/2$ and
\begin{equation}\label{alppa}
\begin{split}
|\Delta_h^1(\partial^\alpha_x\partial^\beta_y \kappa_{j}^{\sigma,\rho}(\cdot,y_0),x_0)|&\le|h||\nabla_x(\partial^\alpha_x\partial^\beta_y \kappa_{j}^{\sigma,\rho})(\xi,y_0)|
\le C|h|2^{j(1-\delta)}|x_0-y_0|^{-n-\delta}.
\end{split}
\end{equation}
Using  the triangle inequality and
Lemma \ref{apu2},
we also have the estimate
\begin{equation}\label{betta}
\begin{split}
|\Delta_h^1(\partial^\alpha_x\partial^\beta_y \kappa_{j}^{\sigma,\rho}(\cdot,y_0),x_0)|
\le
C2^{-j\delta}|x_0-y_0|^{-n-\delta}.
\end{split}
\end{equation}
By using these, we have
\begin{equation}\label{tos}
\begin{split}
&\sum_{j=\ell}^\infty 
|\partial^\alpha_x \partial^\beta_y \kappa_{j}^{\sigma,\rho}(x_0+h,y_0)-\partial^\alpha_x\partial^\beta_y \kappa_{j}^{\sigma,\rho}(x_0,y_0)|\\
&\le C|x_0-y_0|^{-n-\delta}\bigg(|h|\sum_{j=-\infty}^{j_0} 2^{j(1-\delta)} + \sum_{j=j_0}^\infty
2^{-j\delta}\bigg)
\le C|h|^\delta|x_0-y_0|^{-n-\delta}.
\end{split}
\end{equation}
It follows that 
\begin{equation}\label{eku}
|\partial^\alpha_x\partial^\beta_y K(x_0+h,y_0)-\partial^\alpha_x\partial^\beta_y K(x_0,y_0)|\le 
C|h|^{\delta}|x_0-y_0|^{-n-{\delta}},
\end{equation}
if $x_0+h\in \Omega(x_0,y_0)$.

What comes next is to establish this estimate for more general $h$'s. This 
argument proceeds precisely as in \cite[p. 78]{t1dom}, and we omit the details
which are based on the uniformity of the domain.

\subsection{Extension of kernels}
Following theorem gives an extension result for smooth kernels defined on uniform domains.
According to Theorem \ref{fixss} it equally well gives an extension theorem for standard kernels.

\begin{thm}\label{extensionB}
Let $\Omega\subset\R^n$ be a uniform domain,
$0<m<n$, $0<\delta<1$, and
$K\in \mathcal{K}^{-m}_\Omega(\delta)$ be a smooth kernel. Then there
exists 
$\tilde K\in \mathcal{K}^{-m}_{\R^n}(\delta)$ such that 
\[\tilde K|\Omega\times \Omega\setminus \{(x,x)\}=K.\] In words, 
$K$ has an extension to a smooth kernel
$\tilde K:\R^n\times\R^n\setminus\{(x,x)\}\to\C$.
\end{thm}

The proof of this theorem is available in \cite[Theorem 5.56]{t1dom}.
To describe it briefly, one decomposes  smooth kernels by using a partition of unity, subordinate
to the Whitney decomposition of the open set
$\R^n\times \R^n\setminus \{(x,x)\}$.
This 
results in a characterization of smooth kernels in terms of a certain atomic decomposition.
In particular, the extension problem reduces to the H\"older extension of individual kernel atoms, which are
(a priori) H\"older functions defined
on $\Omega\times\Omega$.

\subsection{$T1$ theorem for restricted WSIO's}
A standard kernel $K\in\mathrm{K}^{-m}_{\Omega}(\delta)$ gives a rise to a weakly singular integral operator
(abbreviated WSIO).
Such operators emerge in connection with elliptic PDE's on domains, and
their derivatives of order $m$ should be thought of as singular integral
operators. The boundedness properties of WSIO's are treated in \cite{t1dom}.

Here is the formal definition for WSIO's:
an integral operator $T$ is associated with a standard kernel if there
exists $K\in\mathrm{K}^{-m}_{\Omega}(\delta)$, $m\in\{1,2,\ldots,n-1\}$, such that
\[
Tf(x)=\int_\Omega K(x,y)f(y)dy,\quad x\in\Omega,\quad f\in C^\infty_0(\Omega).
\]
We denote this by $T\in\mathrm{SK}^{-m}_{\Omega}(\delta)$.
The following
theorem describes the boundedness properties of globally defined WSIO's
when restricted to a suitable domain.
A proof 
is in \cite[Theorem 3.118]{t1dom}.

\begin{thm}\label{korits}
Let  $\emptyset\not=\Omega\subset\R^n$ be a $c$-coplump domain such
that either $\Omega=\R^n$ or $\mathrm{diam}(\R^n\setminus\Omega)=\infty$.
Let $T\in\mathrm{SK}^{-m}_{\R^n}(\delta)$, where
$0<m<n$ and
$0<\delta<1$. Then the following two conditions are equivalent
\begin{itemize}
\item 
$T\chi_\Omega,T^*\chi_\Omega\in\dot f^{m,2}_\infty(\Omega)$,
\item $\partial^\sigma T,\partial^\sigma T^*\in\mathcal{L}(L^2(\Omega))$ if $|\sigma|=m$.
\end{itemize}
Furthermore, if  these conditions hold true, then there exists
$\hat T\in\mathrm{SK}^{-m}_{\R^n}(\delta)$ whose associated kernel coincides to that
of $T$ on  $\Omega\times\Omega\setminus\{(x,x)\}$,
\begin{equation}\label{eq}
\langle Tf,g\rangle = \langle \hat Tf,g\rangle,\quad f,g\in C^\infty_0(\Omega),
\end{equation}
the operator $\hat T$ satisfies  two equivalent conditions above with $\Omega=\R^n$,
and operators
$\partial^\sigma \hat T$ and $\partial^\sigma \hat T^*$, $|\sigma|=2$,  are Calder\'on--Zygmund operators.
\end{thm}

The space $\dot f^{m,2}_\infty(\Omega)$ is a $\mathrm{BMO}$-type Sobolev space, see \cite[Definition 3.47]{t1dom}.
Hence Theorem \ref{korits} is similar to the well known result about Calder\'on--Zygmund type
operators: $T1$ theorem due to David and Journ\'e \cite{D-J}, in which
the characterizing conditions for $L^2$ boundedness include that $T1,T^*1\in\mathrm{BMO}(\R^n)$.

The boundedness properties of restricted operators can be used
to study WSIO's that are (a priori) defined on domains. Indeed, by using theorem \ref{extensionB} we can
first extend a given operator $T\in\mathrm{SK}^{-m}_{\Omega}(\delta)$ --
which is associated to a kernel $K\in\mathrm{K}^{-m}_{\Omega}(\delta)$ --
to a globally defined operator $\hat T\in\mathrm{SK}^{-m}_{\R^n}(\delta)$
if the underlying domain is uniform. This extension is given by the formula
\[
\tilde Tf(x)=\int_{\R^n}Ê\tilde K(x,y)f(y)dy,\quad x\in\R^n,\quad f\in C^\infty_0(\R^n).
\]
Because the extended kernel $\tilde K$ coincides with the kernel $K$ 
in $\Omega\times\Omega\setminus\{(x,x)\}$, we see that
$\partial^\sigma \tilde T\in \mathcal{L}(L^p(\Omega))$ if, and only if, $\partial^\sigma T\in\mathcal{L}(L^p(\Omega))$.


The proof of Theorem \ref{korits} depends
on certain reflected paraproduct operators whose role is twofold: they are used in
a reduction as in the proof of David and Journ\'e, but  they also modify the associated kernel $K$ outside
of
the product domain
$\Omega\times\Omega$ to reach a Calder\'on--Zygmund operator $\hat T$ whose second order partials are bounded on $L^2(\R^n)$. 
The novelty
lies in this modification procedure where certain boundary terms are  treated by using 
coplumpness.
\section{Proof of main results}\label{three}

Theorem \ref{dss} is proven by using Schauder theory. A byproduct is that the 
Green's function satisfies the standard kernel estimates, and the Green's operator
is WSIO associated to a standard kernel. Finally the machinery 
in Section \ref{two} is invoked to finish the proof of Theorem \ref{mainth}.

\subsection{Schauder estimates on $C^{2,\alpha}$ domains}

We invoke certain estimates for the solutions of 
second order elliptic equations.
These are classical Schauder estimates involving a (possibly empty) boundary portion \cite{gilbarg}.

\begin{defn}
An open set $D\subset\R^n$ will be said to have a boundary partion
$T\subset \partial D$ (of class $C^{2,\alpha}$) if at
each point $\bar y\in T$ there is a ball $B\subset B(\bar y,\rho)$ which is centered at $\bar y$, 
which satisfies $B\cap\partial D\subset T$, and
in which
the conditions (1)--(3) of Definition \ref{smoothdom} are satisfied. 
\end{defn}

We invoke the following notation from \cite[pp. 95--96]{gilbarg}.
Let $D$ be an open set in $\R^n$ with a boundary partion
$T$ of class $C^{2,\alpha}$. For $x,y\in\Omega$, we set
\[
\bar \delta_x=\mathrm{dist}(x,\partial D\setminus T),\quad \bar \delta_{x,y}=\bar \delta_x\wedge \bar \delta_y.
\]
For bounded continuous functions $u\in C(D)$ we define $|u|_{0,D}=\sup_{x\in D}|u(x)|$ and,
for functions $u\in C^{k,\alpha}(D\cup T)$, we define 
\begin{align*}
[u]^*_{k,0;D\cup T} &= [u]^*_{k;D\cup T}=\sup_{x\in D;|\beta|=k} \bar \delta_x^k |\partial^\beta u(x)|,\quad k=0,1,\ldots \\
[u]^*_{k,\alpha;D\cup T}Ê&= \sup_{x,y\in D;|\beta|=k} \bar \delta_{x,y}^{k+\alpha}\frac{|\partial^\beta u(x)-\partial^\beta u(y)|}{|x-y|^{\alpha}},\quad 0<\alpha \le 1;\\
|u|^*_{k,0;D\cup T}Ê&= |u|^*_{k;D\cup T} = \sum_{j=0}^k [u]^*_{j;D\cup T};\\
|u|^*_{k,\alpha;D\cup T}Ê&= |u|^*_{k;D\cup T} + [u]^*_{k,\alpha;D\cup T};\\
|u|^{(k)}_{0,\alpha,D\cup T} &= \sup_{x\in D} \bar \delta_x^k |u(x)| + \sup_{x,y\in D} \bar \delta_{x,y}^{k+\alpha} \frac{|u(x)-u(y)|}{|x-y|^\alpha}.
\end{align*}
In case $T=\emptyset$ we denote $D\cup T=D$ in the definitions above.

We rely on the following local boundary estimate, see  \cite[Theorem 6.2.]{gilbarg} and
\cite[Lemma 6.4.]{gilbarg}.

\begin{lem}\label{boupar}
Assume that $D\subset \R^n$ is a proper open subset of $\R^n_+$ with  (possibly empty) boundary portion
$T\subset \partial \R^n_+\cap \partial D$. Suppose that $\tilde u\in C^{2,\alpha}(D\cup T)$ is a bounded solution in $D$ of 
\[
\tilde L \tilde u(x)=\sum_{i,j=1}^n \tilde a^{ij}(x)\partial_{ij} \tilde u(x) + \sum_{i=1}^n \tilde b^{i}(x)\partial_i \tilde u(x)  = 0,\quad x\in D,
\]
and it satisfies the boundary condition $\tilde u\equiv 0$ on $T$. We also assume that  $\tilde L$ is strictly elliptic in the sense that there
exists a positive constant $\tilde \lambda>0$ such that
\[
\sum_{i,j=1}^n \tilde a^{ij}(x)\xi_i\xi_j\ge \tilde \lambda|\xi|^2,\quad x\in D,\quad \xi\in\R^n.
\]
Furthermore, if $i,j\in \{1,2,\ldots,n\}$, the coefficients are assumed to satisfy $\tilde a^{ij}=\tilde a^{ji}$ and
\[|\tilde a^{ij}|^{(0)}_{0,\alpha;D\cup T} + |\tilde b^i|^{(1)}_{0,\alpha;D\cup T}\le \tilde \Lambda.\]
Under these assumptions, we have the estimate
\[
|\tilde u|^*_{2,\alpha;D\cup T}\le C|\tilde u|_{0;D},
\]
where $C=C(n,\alpha,\tilde \lambda,\tilde \Lambda)$.
\end{lem}

As a consequence, we obtain the following local boundary estimate for curved boundaries.

\begin{lem}\label{elles}
Let $\Omega\subset \R^n$ be a bounded $C^{2,\alpha}$ domain and $0<S\le \rho$.
Let $\bar y\in \partial\Omega$ and \[B=B(\bar y,S)\cap \Omega,\quad T=B(\bar y,S)\cap \partial\Omega\subsetneq \partial B.\]
Let $u\in C^{2,\alpha}(\bar B)$ satisfy $Lu=0$ on $B$, where $L$ is defined in \eqref{alla},
and assume that $u\equiv 0$ on $T$. Then
\[
|u|^{*}_{2,\alpha;B\cup T} \le C|u|_{0,B},
\]
where $C=C(\alpha,K,\Omega,L)$.
\end{lem}

\begin{proof}
By Definition \ref{smoothdom} there is a neighborhood $N$ of $\bar y$ 
such that $B(\bar y,S)\subset\subset N$,
and a diffeomorphism
$\psi$ defined on $N$ that straightens the boundary near $\bar y$. 
Denote 
\[
D'=\psi(B)\subset\R^n_+,\quad T'=\psi(T)\subset \partial \R^n_+\cap \partial D'.\] Then $T'$
is a boundary portion of $D'$. Under mapping
$y=\psi(x)=(\psi_1(x),\ldots,\psi_n(x))$, $x\in \bar B$, let $\tilde u(y)=u(x)$ and
$\tilde L\tilde u(y)=L u(x)=0$, where
\[
\tilde L\tilde u(y)=\sum_{i,j=1}^n \tilde a^{ij}(y)\partial_{ij}\tilde u(y) + \sum_{i=1}^n \tilde b^i(y)\partial_i \tilde u(y) = 0,\quad y=\psi(x)\in D',
\]
and
\begin{align*}
\tilde a^{ij}(y)&=\sum_{r,s=1}^n a^{rs}(x)\partial_r\psi_i(x)\partial_s \psi_j(x);\\\tilde b^i(y)&=\sum_{r,s=1}^n a^{rs}(x)\partial_{rs}\psi_i(x)+\sum_{r=1}^n \bigg(\sum_{s=1}^n \partial_s a^{sr}(x)\bigg)\partial_r\psi_i(x).
\end{align*}
A straightforward computation using \eqref{elliptic} shows that,
for $y\in D'$,
\begin{align*}
C(K,n)  \lambda|\xi|^2&\le 
\lambda |\nabla(\xi\cdot \psi)(x)|^2\\&\le  \sum_{r,s=1}^n a^{rs}(x)\partial_r (\xi\cdot \psi)(x)\partial_s(\xi\cdot \psi)(x)
=\sum_{i,j=1}^n \tilde a^{ij}(y)\xi_i\xi_j,\quad \xi\in\R^n,
\end{align*}
for a constant $C(n,K)>0$, depending on 
$n$ and on $K$.
Because $a^{ij}\in C^{1,\alpha}(\overline{\Omega})$ and $\mathrm{diam}(D')\le K\mathrm{diam}(B)\le K\mathrm{diam}(\Omega)$, we also have
\[
|\tilde a^{ij}|^{(0)}_{0,\alpha;D'\cup T'} + |\tilde b^i|^{(1)}_{0,\alpha;D'\cup T'} \le C(\alpha,K,\Omega,L).
\]
Furthermore, we have $\tilde u =u\circ \psi^{-1}\in C^{2,\alpha}(\overline{D'})$, so that
the conditions of Lemma \ref{boupar} are satisfied for the
equation $L\tilde u = 0$ in $D'$ with the boundary
portion $T'$. Therefore we have
\begin{align*}
|u|^*_{2,\alpha;B\cup T}&\le C(\alpha,K,\Omega)|\tilde u|^*_{2,\alpha;D'\cup T'}\\&\le C(\alpha,K,\Omega,L)|\tilde u|_{0,D'}=C(\alpha,K,\Omega,L)|u|_{0,B},
\end{align*}
where in the first inequality can be found in \cite[p. 96]{gilbarg}
This is the required estimate. 
\end{proof}

\subsection{Proof of Theorem \ref{dss}}
First we establish qualitative
$C^{2,\alpha}$ estimates
for the Green's function up to the boundary. In this connection
we advance quite rapidly, providing citations to the required regularity results.
Then we proceed to quantitative estimates, where
the prior Schauder estimates are used. The proof
of Theorem \ref{dss} is finished by local-to-global type H\"older estimate.

To begin with, by using Theorem \ref{gu}, we find
that the function
\[
u=G(x,\cdot),\quad x\in\Omega,
\]
belongs to $C(\partial\Omega\cup \Omega\setminus\{x\})$ if we define
$u(y)=0$ in points $y\in\partial\Omega$.
The required exterior sphere condition is satisfied in our situation:
by using the implicit function theorem, it follows that $\partial\Omega$ can be locally represented
as graph of a $C^{2,\alpha}$ function. Then the exterior sphere property follows from \cite[Lemma 2.2]{aikawa}.

Denote $\Omega(r)=\Omega\setminus B(x,r)$ if $r>0$ is so small that $B(x,r)\subset\subset\Omega$. 
According to Theorem \ref{gext}, function $u$ belongs to the Sobolev space $W^{1,2}(\Omega(r))$, and it is a weak solution
to the equation $Lu=0$ in the domain $\Omega(r)$. That is, it satisfies
\[
\int_{\Omega(r)}Êa^{ij}(y)\partial_j u(y)\partial_i \phi(y)dy = 0,\quad \phi\in C^1_0(\Omega(r)).
\]
The coefficients $a^{ij}$ belong to the space $C^{1,\alpha}(\overline{\Omega})$. 
In particular, they are bounded and uniformly Lipschitz continuous in $\Omega$. Hence, by
using \eqref{elliptic} and Theorem \cite[Theorem 8.8]{gilbarg}, we find that
$u\in W^{2,2}_{\loc}(\Omega(r))$.
Furthermore $u$ is a strong solution to the equation $Lu=0$:
\begin{equation}\label{repr}
Lu = \sum_{i,j=1}^n a^{ij}\partial_{ij} u + \sum_{i=1}^n \bigg(\sum_{j=1}^n \partial_j a^{ji}\bigg) \partial_i u = 0
\end{equation}
almost everywhere in $\Omega(r)$. Since the coefficients
in \eqref{repr} belong to $C^{0,\alpha}(\overline{\Omega(r)})$, we have
the regularity $u\in C^{2,\alpha}(\Omega(r))$ by Theorem \cite[Theorem 9.19]{gilbarg}.

From the discussion above it is now clear that
\[
u\in C(\overline{\Omega(r)})\cap C^{2,\alpha}(\Omega(r))
\]
is a classical solution to the equation \eqref{repr} in $\Omega(r)$ with boundary values $u=0$ in  
the $C^{2,\alpha}$ boundary portion $\partial\Omega\subset \partial\Omega(r)$. By using
\cite[Lemma 6.18]{gilbarg} we then deduce
that $u\in C^{2,\alpha}(\Omega(r)\cup \partial\Omega)$ and, because $r>0$ was
arbitrary, we find that $u\in C^{2,\alpha}(\Omega\setminus\{x\}\cup \partial \Omega)$
satisfies the equation \eqref{repr} pointwise in $\Omega\setminus\{x\}$

To conclude from above and by using Theorem \ref{gu}, the Green's function has
 the following properties in case $\Omega\subset\R^n$, $n\ge 3$, is a bounded $C^{2,\alpha}$ domain
 and the coefficients $a^{ij}$ belong to the space $C^{1,\alpha}(\overline{\Omega})$:
\begin{equation}\label{impreg}
\begin{cases}
L \lbrace G(x,\cdot)\rbrace(y)=0,\quad &x\in\Omega,y\in\Omega\setminus\{x\};\\
G(x,\cdot)\in C^{2,\alpha}(\partial\Omega\cup \Omega\setminus\{x\});\quad &x\in\Omega;\\
G(x,y)=0,\quad &x\in\Omega,y\in\partial\Omega;\\
G(x,y)\le C_L|x-y|^{2-n}\min\{1,\delta(x)|x-y|^{-1}\},\quad &x\in\Omega ,y\in\Omega\setminus\{x\}.
\end{cases}
\end{equation}

We are ready for the main parts of Theorem \ref{dss}.


\begin{lem}\label{gest}
The following size-estimate is valid for disjoint points $x,y\in\Omega$,
\begin{equation}\label{redsize}
|\partial^\beta_yG(x,y)|\le C|x-y|^{2-n-|\beta|}\min\bigg\{1,\frac{\delta(x)}{|x-y|}\bigg\},\quad |\beta|\le 2.
\end{equation}
If, in addition $|\beta|=2$ and $y+h\in B(y,\mathrm{dist}(y,\partial\Omega)\wedge \rho|x-y|/8\mathrm{diam}(\Omega))$, then
\begin{equation}\label{redhol}
|\partial^\beta_y G(x,y+h)-\partial^\beta_y G(x,y)|\le C|h|^\alpha |x-y|^{-n-\alpha}\min\bigg\{1,\frac{\delta(x)}{|x-y|}\bigg\}.
\end{equation}
Here $C=C(\alpha,K,\Omega,L)$.
\end{lem}

\begin{proof}
Throughout the proof $C$ denotes a constant, which
depends at most on the parameters $\alpha,K,\Omega,L$.
We also denote by \[\Gamma_{\Omega}(x,y):=|x-y|^{2-n}\min\{1,\delta(x)|x-y|^{-1}\}\]
the upper bound  in \eqref{impreg}.
Let $x_0,y_0\in\Omega$. We will treat the following cases
\begin{equation}\label{cases}
\frac{|x_0-y_0|}{\mathrm{dist}(y_0,\partial\Omega)}\le \frac{4\mathrm{diam}(\Omega)}{\rho},\quad \frac{|x_0-y_0|}{\mathrm{dist}(y_0,\partial\Omega)}\ge \frac{4\mathrm{diam}(\Omega)}{\rho}
\end{equation}
separately.

We begin with the first case in \eqref{cases}, which is equivalent to that
\begin{equation}\label{case1}
\frac{\rho}{4\mathrm{diam}(\Omega)}|x_0-y_0|\le \mathrm{dist}(y_0,\partial\Omega).
\end{equation}
Denote $R=\rho|x_0-y_0|/8\mathrm{diam}$. We claim that
\begin{equation}\label{es1}
B(y_0,2R)\subset \Omega,\quad \mathrm{dist}(x_0,B(y_0,2R))> \frac{15}{16}|x_0-y_0|>0.
\end{equation}
Notice that the  inclusion in \eqref{es1} follows from that, if $y\in B(y_0,2R)$, then
\begin{align*}
|y-y_0|< 2R= \frac{\rho}{4\mathrm{diam}(\Omega)}|x_0-y_0|\le \mathrm{dist}(y_0,\partial\Omega)
\end{align*}
by using \eqref{case1}. Furthermore, by using estimate $\rho/\mathrm{diam}(\Omega)< 1/4$, we get
\begin{align*}
|x_0-y|&\ge |x_0-y_0|-|y_0-y|\\&> |x_0-y_0|-2R=
|x_0-y_0|-\frac{\rho}{4\mathrm{diam}(\Omega)}|x_0-y_0|\\&> |x_0-y_0|-\frac{1}{16}|x_0-y_0|= \frac{15}{16}|x_0-y_0|.
\end{align*}
The inequality in \eqref{es1} follows by infimizing the left-hand side over $y\in B(y_0,2R)$.

Notice that
$D:=B(y_0,2R)\subset\Omega\setminus\{x_0\}$ by \eqref{es1}.
By using  \eqref{impreg}, we find that the function $u=G(x_0,\cdot)$ satisfies
$Lu=0$ in $D$ and $|u(y)|\le C\Gamma_\Omega(x_0,y)$ for $y\in D$. Without loss of generality, we can assume that $D\subset \R^n_+$.
Hence, by using Lemma \ref{boupar} with $D\subset \R^n_+$ and $T=\emptyset$, we obtain the important estimate
 \begin{equation}\label{simmsa}
 \begin{split}
|u|^*_{2,D}  + [u]^*_{2,\alpha;D}&= |u|^*_{2,\alpha;D}\le C|u|_{0,D}= C \sup_{y\in D} |u(y)|\\&=C \sup_{y\in D} |G(x_0,y)|\le 
C \sup_{y\in D} \Gamma_\Omega(x_0,y_0)\le C\Gamma_\Omega(x_0,y_0).
\end{split}
 \end{equation}
  It remains to collect the implications of this strong estimate.
 The first consequence of \eqref{simmsa} is that, if $y\in B(y_0,R)\subset D$ and $|\beta|\le 2$, we have
 \begin{equation}\label{sas}
 \begin{split}
{\bar d}_y^{|\beta|}|\partial^\beta_y G(x_0,y)|=\bar \delta_y^{|\beta|}|\partial^\beta u(y)|
&\le\sup_{z\in D;|\gamma|=|\beta|} \bar \delta_z^{|\beta|} |\partial^\gamma u(z)| = [u]^*_{|\beta|;D}
\\& \le \sum_{j=0}^2 [u]^*_{j;D}=|u|^*_{2;D}\le C\Gamma_\Omega(x_0,y_0).
\end{split}
 \end{equation}
By the inclusion in \eqref{es1}, we have $\bar \delta_y=\mathrm{dist}(y,\partial D)\ge R\ge C|x_0-y_0|$ given that  $y\in B(y_0,R)$.
Hence  estimate \eqref{sas} implies  that
\[
|\partial^\beta_y G(x_0,y)|\le C\Gamma_\Omega(x_0,y_0)|x_0-y_0|^{-|\beta|},\quad y\in B(y_0,R).
\]
In the special case $y=y_0$ this implies the required estimate \eqref{redsize}.

Next assume that $|\beta|=2$ and $y_0+h\in B(y_0,R)\subset D$. Then,
by \eqref{simmsa}, we have
\begin{align*}
&\min\{\bar \delta_{y_0},\bar \delta_{y_0+h}\}^{2+\alpha}\frac{|\partial^\beta_y G(x_0,y_0+h)-\partial^\beta_y G(x_0,y_0)|}{|h|^\alpha}\\
&=\min\{\bar \delta_{y_0},\bar \delta_{y_0+h}\}^{2+\alpha}\frac{|\partial^\beta u(y_0+h)-\partial^\beta u(y_0)|}{|h|^\alpha}\\
&\le \sup_{z,w\in D;|\beta|=2}\bigg[\min\{\bar \delta_{z},\bar \delta_{w}\}^{2+\alpha}\frac{|\partial^\beta u(z)-\partial^\beta u(w)|}{|z-w|^\alpha}\bigg]=[u]^*_{2,\alpha;D}\le C\Gamma_\Omega(x_0,y_0).
\end{align*}
As above, by using \eqref{es1}, we have the estimate
\[
\min\{\bar \delta_{y_0},\bar \delta_{y_0+h}\}\ge R =C|x_0-y_0|
\]
since $y_0,y_0+h\in B(y_0,R)$.
As a consequence, we obtain the estimate
\[
|\partial^\beta_y G(x_0,y_0+h)-\partial^\beta G(x_0,y_0)|\le C|h|^\alpha \Gamma_\Omega(x_0,y_0)|x_0-y_0|^{-2-\alpha},
\]
 which clearly suffices
for \eqref{redhol}. This concludes
the first case in \eqref{cases}.



Next we proceed to the second case
in \eqref{cases}. This is a boundary estimate, where we assume that
\begin{equation}\label{case2}
\frac{\rho}{4\mathrm{diam}(\Omega)}|x_0-y_0|\ge \mathrm{dist}(y_0,\partial\Omega)=\delta(y_0).
\end{equation}
Denote $S=\rho|x_0-y_0|/\mathrm{diam}(\Omega)$
and fix a point $\bar y\in\partial\Omega$ such
that $|\bar y-y_0|=\delta(y_0)$. 
We begin by claiming that the following auxiliary estimates
\begin{equation}\label{distok}
\mathrm{dist}(x_0,B(\bar y,S))\ge \frac{1}{2}|x_0-y_0|,\quad 
\mathrm{dist}(B(y_0,\delta(y_0)),\R^n\setminus B(\bar y,S)) \ge S/2.
\end{equation}
hold true. Indeed, the first estimate \eqref{distok}
follows from that, if $z\in B(\bar y,S)$, then
\begin{align*}
|x_0-z|&\ge |x_0-y_0|-|y_0-\bar y|-|\bar y-z|\\&\ge |x_0-y_0| - \delta(y_0) -S
\ge |x_0-y_0|-\frac{1}{2}|x_0-y_0|\ge \frac{1}{2}|x_0-y_0|
\end{align*}
because $2\rho/\mathrm{diam}(\Omega)<1/2$, so that $\delta(y_0)+S\le 2\rho|x_0-y_0|/\mathrm{diam}(\Omega)\le |x_0-y_0|/2$.
For the second estimate in  \eqref{distok}, we fix $w\in B(y_0,\delta(y_0))$. Then
\[|w-\bar y|\le |w-y_0|+|y_0-\bar y|\le 2\delta(y_0)\le S/2.\]
If also $z\in \R^n\setminus B(\bar y,S)$, then
\[
|z-w|=|z-\bar y+\bar y-w|\ge |z-\bar y|-|\bar y-w|\ge  S/2.
\]
It remains to infimize the  left-hand side over
 $z$ and $w$.

Later we will invoke Lemma \ref{elles}. For this purpose, we
denote $B=B(\bar y,S)\cap \Omega$ and $T=B(\bar y,S)\cap \partial\Omega\subsetneq \partial B$.
 Notice that, by \eqref{distok}, \[B(y_0,\delta(y_0))=B(y_0,\delta(y_0))\cap \Omega\subset B.\]
Denote  $\bar \delta_y=\mathrm{dist}(y,\partial B\setminus T)$. Then
$\partial B\setminus T\subset \R^n\setminus B(\bar y,S)$ so that, for  $y\in B(y_0,\delta(y_0))$, we have
\begin{equation}\label{dest}
\begin{split}
\bar \delta_y\ge\mathrm{dist}(B(y_0,\delta(y_0)),\partial B\setminus T)&\ge \mathrm{dist}(B(y_0,\delta(y_0)),\R^n\setminus B(\bar y,S))\\&\ge S/2
= C |x_0-y_0|.
\end{split}
 \end{equation}
 by \eqref{distok}.
The function $u=G(x_0,\cdot)$ satisfies $Lu=0$ in $B\subset\subset \bar \Omega\setminus\{x_0\}$, which is seen by using both
\eqref{distok} and \eqref{impreg}. It also satisfies the boundary
 condition $u\equiv 0$ on $T$. As a consequence of Lemma \ref{elles}
 and estimate $|u(y)|\le C_L|x_0-y|^{2-n}$ combined with \eqref{distok}, we have
 \begin{equation}\label{simms}
 \begin{split}
|u|^*_{2,B\cup T}  + [u]^*_{2,\alpha;B\cup T}&\le C|u|_{0,B}\le C  \sup_{y\in B} \Gamma_\Omega(x_0,y)\le C \Gamma_\Omega(x_0,y_0).
\end{split}
 \end{equation}
 The first  consequence of \eqref{simms} is that, if $y\in B(y_0,\delta(y_0))\subset B$ and $|\beta|\le 2$, we have
 \begin{align*}
{\bar d}_y^{|\beta|}|\partial^\beta_y G(x_0,y)|&=\bar \delta_y^{|\beta|}|\partial^\beta u(y)|
\le |u|^*_{2;B\cup T}\le C\Gamma(x_0,y_0).
 \end{align*}
 Taking estimate \eqref{dest} into account, we have
\[
|\partial^\beta_y G(x_0,y)|\le C\Gamma(x_0,y_0)|x_0-y_0|^{-|\beta|},\quad y\in B(y_0,\delta(y_0)).
\]
In the special case $y=y_0$ this yields \eqref{redsize}.

Next, if $y_0+h\in B(y_0,\delta(y_0))\subset B$ and $|\beta|=2$, then
by using \eqref{simms} we have
\begin{align*}
&\min\{\bar \delta_{y_0},\bar \delta_{y_0+h}\}^{2+\alpha}\frac{|\partial^\beta_y G(x_0,y_0+h)-\partial^\beta_y G(x_0,y_0)|}{|h|^\alpha}\le [u]^*_{2,\alpha;B\cup T}\le C\Gamma(x_0,y_0).
\end{align*}
On the other hand, by using \eqref{dest}, we have 
$\min\{\bar \delta_{y_0},\bar \delta_{y_0+h}\}\ge C|x_0-y_0|$.
As a consequence, we find that
\[
|\partial^\beta_y G(x_0,y_0+h)-\partial^\beta G(x_0,y_0)|\le C |h|^\alpha \Gamma_\Omega(x_0,y_0)|x_0-y_0|^{-2-\alpha},
\]
which clearly suffices
for \eqref{redhol}.
\end{proof}

To conclude the proof of Theorem \ref{dss} we still need the global H\"older estimate
\eqref{hol} where, in contrast to Lemma \ref{gest}, point $y+h$ can not be restricted
to the ball 
$B(y,\delta(y)\wedge \rho|x-y|/8\mathrm{diam}(\Omega))$. Proof proceeds as in \cite[p. 78]{t1dom} with minor modifications, and we omit the details that are based on uniformity.

\subsection{Proof of Theorem \ref{mainth}}
The main step is to show that $G\in\mathrm{K}^{-2}_{\Omega}(\alpha)$.
By using \eqref{redsize} we see that 
the first required estimate \eqref{k1} holds true, that is, 
\[|G(x,y)|\le C_G|x-y|^{2-n}\]
if $x,y\in\Omega$ are distinct points.
Next we verify \eqref{k2}. For this purpose we
fix $x\in\Omega$ and let $B=B(y^B,r)\subset\subset\Omega$ be a ball which satisfies
the condition 
\[
\frac{8r\mathrm{diam}(\Omega)}{\rho}\le |x-y^B|.
\]
Then, in particular, we have \[B(y^B,r)\subset B(y^B,\mathrm{dist}(y^B,\partial\Omega)\wedge \rho|x-y^B|/8\mathrm{diam}(\Omega)).\]
Hence, if $\{y,\ldots,y+3h\}\subset B\subset\subset\Omega$, using an integral representation
\cite[p.102]{t1dom}
for the second order differences, we find that
\begin{align*}
|\Delta_h^{3}(G(x,\cdot),B,y)| &\lesssim |h|^2 \sup_{\theta\in [0,2];|\alpha|=2}
\bigg\{|\partial^\alpha_y G(x,y+(1+\theta)h)-\partial^\alpha_y G(x,y+\theta h)|\bigg\}.
\end{align*}
Notice that $y+\theta h,y+(1+\theta)h,y^B\in B$ inside the supremum so that, by invoking Lemma \ref{gest}, we find that
\[
|\Delta_h^{3}(G(x,\cdot),B,y)|\lesssim |h|^{2}\mathrm{diam}(B)^{\alpha}|x-y^B|^{-n-\delta}\le \mathrm{diam}(B)^{2+\alpha}|x-y^B|^{-n-\alpha}.
\]
Integrating this inequality over $y\in B$ shows that $G$ satisfies
the  estimate \eqref{k2}. Because $G$ is symmetric, we also find
that the kernel $(x,y)\mapsto G(y,x)=G(x,y)$ also satisfies this condition. 

All in all,
we have shown that $G\in\mathrm{K}^{-2}_{\Omega}(\alpha)$. Also, by Theorem \ref{smoothun},
the bounded $C^{2,\alpha}$ domain $\Omega$ is uniform.
Hence 
Theorem \ref{fixss} applies and,
as a consequence, we have the following.

\begin{thm}\label{firstes}
Let $\Omega\subset\R^n$, $n\ge 3$, be a bounded $C^{2,\alpha}$ domain
and assume that the coefficients $a^{ij}$ of $L$ belong to $C^{1,\alpha}(\overline{\Omega})$.
Then Green's function ${G}$ of $L$ in $\Omega$ is a smooth kernel, that is,
$G\in\mathcal{K}^{-2}_\Omega(\alpha)$.
\end{thm}

By combining
theorems \ref{firstes} and \ref{extensionB}, we
see that there exists a globally defined smooth kernel
$\tilde G\in\mathcal{K}^{-2}_{\R^n}(\alpha)=\mathrm{K}^{-2}_{\R^n}(\alpha)$ such that
\[\tilde G|\Omega\times\Omega\setminus\{(x,x)\}=G.\] Next we invoke Theorem \ref{smoothun} for the conclusion that
$\Omega$ is $c$-coplump for some $c\ge 1$. Also, since $\Omega$ is bounded,
we have $\mathrm{diam}(\R^n\setminus\Omega)=\infty$. 
Furthermore, the operators $\partial^\sigma\mathcal{G}=\partial^\sigma \mathcal{G}^*$, $|\sigma|=2$,
are bounded on $L^2(\Omega)$ \cite[Theorem 8.12]{gilbarg}. Hence the assumptions
of Theorem \ref{korits} hold true and, by invoking it, we finish the proof
Theorem \ref{mainth}.

\section*{Acknowledgements}

Author thanks Prof. Kari Astala, Jan Cristina, and Dr. Aleksi V\"ah\"akangas for encouraging discussions.

\end{document}